\documentclass{article}
\usepackage{amsfonts}
\usepackage{amsmath}

\newtheorem{theorem}{Theorem}

\newtheorem{example}[theorem]{Example}

\newtheorem{lemma}[theorem]{Lemma}

\newtheorem{proposition}[theorem]{Proposition}
\newtheorem{remark}[theorem]{Remark}

\newenvironment{proof}[1][Proof]{\textbf{#1.} }{\ \rule{0.5em}{0.5em}}

\newtheorem{hypothesis}[theorem]{Hypothesis}

\begin{document}

\title{Optimal Gaussian density estimates for a class of stochastic equations with additive noise}
\author{
  \\
  {David Nualart}    \thanks{Supported by the NSF grant DMS-0904538.}         \\
  {\small \it Department of Mathematics }  \\[-0.15cm]
  {\small \it University of Kansas}          \\[-0.15cm]
  {\small \it Lawrence, Kansas, 66045, USA}  \\[-0.15cm]
  { \small {\tt nualart@math.ku.edu}}
  \\[-0.1cm]
\and
\\
  {Llu\'{\i}s Quer-Sardanyons \thanks{Supported by the grant MEC-FEDER Ref. MTM20006-06427 from the
  Direcci\'on General de Investigaci\'on, Ministerio de Educaci\'on y Ciencia, Spain.}}              \\
  {\small\it Departament de Matem\`atiques} \\[-0.15cm]
  {\small\it  Universitat Aut\`onoma de Barcelona}          \\[-0.15cm]
  {\small\it 08193 Bellaterra (Barcelona), Spain}  \\[-0.15cm]
  {\small  {\tt quer@mat.uab.cat }} \\[-0.1cm]
  \\}

\maketitle

\begin{abstract}
In this note, we establish optimal lower and upper Gaussian bounds for the density of 
the solution to a class of stochastic integral equations driven by an additive 
spatially homogeneous Gaussian random field. The proof is based on
the techniques of the Malliavin calculus and a
density formula obtained by Nourdin and Viens in \cite{nv}. 
Then, the main result is applied to 
the mild solution of a general class of SPDEs driven by a Gaussian noise which is white in 
time and has a spatially homogeneous correlation. In particular, this covers the case of 
the stochastic heat and wave equations in $\mathbb{R}^d$ with $d\geq 1$ and $d\leq 3$, respectively.
The upper and lower Gaussian bounds have the same form and are given in terms of the 
variance of the stochastic integral term in the mild form of the equation.

\end{abstract}

\bigskip

\bigskip


\medskip


\newpage

\section{Introduction and main result}
Let $X=\{X(t,x), \; (t,x)\in [0,T]\times \mathbb{R}^{d}\}$ be a zero mean
Gaussian process with an homogeneous covariance function of the form%
\[
E(X(t,x)X(s,y))=\Phi (t,s,x-y),
\]%
where we assume that $t\mapsto \Phi(t,t,0)$ is continuous,  $\Phi (t,t,0)>0$ for any $t>0$, and \ $\Phi (t,0,x)=0$
for all $(t,x)\in [0,T]\times \mathbb{R}^d$.

The purpose of this note is to establish optimal Gaussian lower and upper bounds for the probability
density of the solution $u(t,x)$ to the following stochastic integral equation:
\begin{equation}
u(t,x)=X(t,x)+\int_{0}^{t}\int_{\mathbb{R}^{d}}b(u(s,x-y))\Gamma (t-s,dy)ds,
\label{2}
\end{equation}
where $(t,x)\in [0,T]\times \mathbb{R}^d$ and $\Gamma (t,dy)$ is a nonnegative measure such that
\begin{equation}
\sup_{0\leq t\leq
T}\Gamma (t,\mathbb{R}^{d})< +\infty.
\label{gamma}
\end{equation}
Note that if $b:\mathbb{R} \rightarrow \mathbb{R}$ is Lipschitz continuous, there
exists a unique solution of (\ref{2}) and the process $u$ is also homogeneous is the space variable.

In particular, these equations include mild solutions to a large class of SPDEs with additive noise of the form
\begin{equation}
L u(t,x) =  b(u(t,x))+ \sigma \dot{W}(t,x), \qquad (t,x)\in [0,T]\times \mathbb{R}^d,
\label{1}
\end{equation}
where $L$ denotes a second order differential operator with constant coefficients,
the random perturbation $\dot W(t,x)$ stands for a Gaussian noise which is white in time and has a spatially homogeneous correlation,
and $\sigma$ is constant. We assume here vanishing initial conditions. Let us denote by $\mu$ the spectral measure associated to the noise $\dot W$ 
(for the precise description of the noise, see Section \ref{appl-spde}). The process $X(t,x)$ is in this case
\[
X(t,x)= \sigma \int_0^t \int_{\mathbb{R}^{d}} \Gamma(t-s,x-y) W(ds,dy),
\]
where $\Gamma$ denotes the fundamental solution associated to $L$ and, if we simply denote $\Phi(t):=\Phi(t,t,0)$, then
\[
\Phi(t)=\int_0^t \int_{\mathbb{R}^{d}} |\mathcal{F} \Gamma(s)(\xi)|^2 \mu(d\xi) ds,
\]
This has to be a finite quantity for $X(t,x)$ to be well-defined. Let us point out that the main examples of
SPDEs of the form (\ref{1}) to which our results will apply are the stochastic heat equation in any space dimension and
the stochastic wave equation in $\mathbb{R}^d$ with $d\leq 3$. Indeed, these types of SPDEs have been widely studied
during the last two decades, see e.g. \cite{walsh,cn,MS,dalang,pz1,peszat}.

The main result of the paper is the following:
\begin{theorem}\label{gaussian}
Assume that $\Gamma$ satisfies (\ref{gamma}) and  $b$ is of class $\mathcal{C}^1$ with a bounded derivative.
Then, there exists $T_0>0$ such that for all $(t,x)\in (0,T_0)\times \mathbb{R}^d$, 
the random variable $u(t,x)$ solution to (\ref{2}) has a density $p$ satisfying, for almost
every $z\in \mathbb{R}$:
\[
\frac{E\left| u(t,x)-m\right|}{ C_2 \Phi(t)} \exp\left\{ -\frac{(z-m)^2}{C_1 \Phi(t) }
\right\}\leq p(z) \leq \frac{E\left| u(t,x)-m\right|}{ C_1 \Phi(t)} \exp\left\{ -\frac{(z-m)^2}{C_2 \Phi(t)}\right\},
\]
where $m=E(u(t,x))$ and the constants $C_1, C_2$ are positive and only depend on $b$  and $T_0$.
\end{theorem}

In order to prove Theorem \ref{gaussian}, we will apply the techniques of the Malliavin calculus. Indeed, as it has
been also done in \cite{nq2}, we will make use of the recent results obtained by Nourdin and Viens in \cite{nv}.
In this paper, the authors provide sufficient conditions on a one-dimensional Wiener functional
(that is, on a random variable on the Wiener space) so that its law has a density and it is given by an explicit
formula (see \cite[Theorem 3.1 and Corollary 3.3]{nv}). As it will be made precise in the proof of Theorem \ref{gaussian}
(see Section \ref{main}), the application of the above-mentioned result of Nourdin and Viens will require a careful
analysis of the Malliavin derivative of the solution of equation (\ref{2}).

We should also remark that recently, there have been an increasing interest in applying the techniques of the Malliavin
calculus in order to establish Gaussian lower bounds for the probability density of a general class of Wiener
functionals. In particular, this has been applied to diffusion processes and solutions to SPDEs, and the article by
Kusuoka and Stroock \cite{ks} can be considered as the starting point of this methodology. Therein, a Gaussian type
lower bound for the density of a uniformly hypoelliptic diffusion whose drift is a smooth combination of its diffusion
coefficient was established. Later on, Kohatsu-Higa \cite{k2} got rid of that dependence constraint on the drift and,
moreover, took the ideas of Kusuoka and Stroock in order to construct a general method to prove that the density of a
multidimensional functional of the Wiener sheet in $[0,T]\times \mathbb{R}^d$ admits Gaussian lower bounds (see
\cite{k}). Then, in the latter paper, Kohatsu-Higa dealt with a one-dimensional stochastic heat equation in $[0,1]$
(for a related result, see also \cite[Theorem 3.1]{nq2}) and his result was also applied by Dalang and Nualart
\cite{dn} in the case of a one-dimensional reduced wave equation. The ideas of \cite{k} have also been further
developed by Bally \cite{bally} to obtain Gaussian lower bounds for locally elliptic It\^o processes. Eventually,
another recent  method for deriving Gaussian lower bounds for multidimensional Wiener functionals has been obtained by
Malliavin and Nualart \cite{mn} (see \cite{enu} for a one-dimensional version of this result). This method is based on
an exponential moment condition on the divergence of a covering vector field associated to the Wiener functional.

The paper is organised as follows. In Section \ref{pre}, we will briefly recall the main tools of the Malliavin
calculus needed in the proof of the main result. In particular, we will recall the main points of the method by Nourdin
and Viens \cite{nv}. Section \ref{main} is devoted to prove Theorem \ref{gaussian}. Finally, in Section \ref{appl-spde}, 
we apply Theorem \ref{gaussian} to the solution of a general spatially homogeneous SPDE.

\section{Preliminaries}
\label{pre}

In this section, we will briefly describe the Gaussian setting in which we will apply the techniques of the Malliavin
calculus as well as the method of Nourdin and Viens \cite{nv}. For a more complete account on these methodologies, we
refer the reader to \cite{nualart} and \cite[Sections 2 and 3]{nv}, respectively.

We denote by $\mathcal{H}$ the  Gaussian space generated by the Gaussian process $\{X(t,x), \; (t,x)\in [0,T]\times\mathbb{R}^d\}$; 
for a complete account on Gaussian Hilbert spaces, we refer the reader to \cite[Chapter 1]{janson}.

Then, as usual, we denote by $D$ the Malliavin derivative, defined as a closed and unbounded operator from
$L^2(\Omega)$ into $L^2(\Omega; \mathcal{H})$, whose domain is denoted by $\mathbb{D}^{1,2}$. Thus, for any  random
variable $F$ belonging to $\mathbb{D}^{1,2}$, its Malliavin derivative $DF$ defines an element in
$L^2(\Omega;\mathcal{H})$.

Another important operator in the Malliavin calculus theory that plays an important role in \cite{nv} is the  generator
of the Orstein-Uhlenbeck semigroup (see \cite[Section 1.4]{nualart}). It is usually denoted by $L$ and can be defined
by means of its Wiener chaos expansion:
$$L=\sum_{n=0}^\infty -n J_n,$$
where $J_n$ denotes the projection onto the $n$th Wiener chaos.

In \cite{nv}, the authors consider a random variable $F\in \mathbb{D}^{1,2}$ with mean zero and define
the following function in $\mathbb{R}$:
$$g_F(z):=E [ \langle DF,-DL^{-1} F\rangle_\mathcal{H} |F=z],$$
where $L^{-1}$ denotes the pseudo-inverse of the generator of the Orstein-Uhlenbeck semigroup $L$. We observe that, by
\cite[Proposition 3.9]{np}, it holds $g(z)\geq 0$ on the support of $F$.  Then, Nourdin and Viens prove the following
result (see \cite[Theorem 3.1 and Corollary 3.3]{nv}):
\begin{theorem}\label{nv}
Assume that there exists a positive constant $c_1$ such that
$$g_F(F)\geq c_1,\quad \text{a.s.}$$
Then, the law of $F$ has a density $\rho$ whose support is $\mathbb{R}$ and satisfies,
almost everywhere in $\mathbb{R}$:
\[
\rho(z)=\frac{E|F|}{2g_F(z)}\exp\left(-\int_0^z \frac{y}{g_F(y)}dy\right).
\]
\end{theorem}
As stated in \cite[Corollary 3.5]{nv}, an immediate consequence of the above theorem is that,
if one also has that $g_F(F)\leq c_2$, a.s., then the density $\rho$ satisfies, for almost all $z\in
\mathbb{R}$:
\[
\frac{E|F|}{2c_1}\exp\left(-\frac{z^2}{2c_2}\right)\leq
\rho(z)\leq \frac{E|F|}{2c_2}\exp\left(-\frac{z^2}{2c_1}\right).
\]
In order to deal with particular applications of this method, \cite[Proposition 3.7]{nv} establishes an alternative
formula for $g_F(F)$. Indeed, given a random variable  $F\in \mathbb{D}^{1,2}$, one can write $DF=\Phi_F(W)$, where
$\Phi_F$ is a measurable mapping from $\mathbb{R}^\mathcal{H}$ to $\mathcal{H}$, determined $(P\circ W^{-1})$-almost
surely (see \cite{nualart}, p. 54-55). Then, it holds that
\begin{equation}
g_F(F)=\int_0^\infty e^{-\theta} \textbf{E}\left[ \langle
\Phi_F(W),\Phi_F(e^{-\theta}W+\sqrt{1-e^{-2\theta}}W')\rangle_\mathcal{H}
\big| F\right] d\theta,
\label{23}
\end{equation}
where $W'$ stands for an independent copy of $W$ such that $W$ and $W'$ are defined on
the product probability space $(\Omega\times \Omega',\mathcal{F}\otimes \mathcal{F}',P\times P')$. Eventually,
$\textbf{E}$ denotes the mathematical expectation with respect to $P\times P'$.

Let us observe that formula (\ref{23}) can be still rewritten in the following form:
\[
g_F(F)= \int_0^\infty  e^{-\theta} E\left[ E'\left( \left\langle
DF,\widetilde{DF}\right\rangle_\mathcal{H}\right) \big| F \right] d\theta,
\]
where, for any random variable $X$ defined in
$(\Omega,\mathcal{F},P)$, $\widetilde{X}$ denotes the shifted random
variable in $\Omega\times \Omega'$
\[
\widetilde{X}(\omega,\omega')=X(e^{-\theta}\omega+
\sqrt{1-e^{-2\theta}}\omega'),\; \omega\in \Omega,\; \omega'\in
\Omega.
\]
Notice that, indeed, $\widetilde{X}$ depends on the
parameter $\theta$, but we have decided to drop its explicit
dependence for the sake of simplicity.

\section{Proof of the main result}
\label{main}

This section is devoted to prove Theorem \ref{gaussian}.
To begin with, we have the following result, whose proof is straightforward and omitted.

 \begin{proposition}
Suppose that $\Gamma$ satisfies condition (\ref{gamma}) and $b:\mathbb{R}\rightarrow \mathbb{R}$ is of class
$\mathcal{C}^{1}$ with a bounded derivative. Let $u(t,x)$ be the solution to Equation (\ref{2}). Then $u(t,x)$ belongs
to the space $\mathbb{D}^{1,2}$ and
\begin{equation}  \label{3}
Du(t,x)=X(t,x)+\int_{0}^{t}\int_{\mathbb{R}^{d}}b^{\prime
}(u(s,x-y))Du(s,x-y)\Gamma (t-s,dy)ds,
\end{equation}
a.s. for all $(t,x)\in (0,T]\times \mathbb{R}^d$.
  \end{proposition}
Let us remark that the pathwise integral in (\ref{3}) takes values in the Hilbert space $\mathcal{H}$ and can be defined using a standard procedure
(see e.g. \cite[p. 292]{nq}).

  Owing to Theorem \ref{nv}, Theorem \ref{gaussian} will be a consequence of the following
proposition. For  $(t,x)\in (0,T]\times \mathbb{R}^d$ we set $F:=u(t,x)-E(u(t,x))$, so that we need to find almost sure upper and lower bounds
for the random variable $g_F(F)$:
\begin{equation}
g_F(F) =\int_0^\infty e^{-\theta} E\left[ E'\left( \langle
Du(t,x),\widetilde{Du(t,x)}\rangle_{\mathcal{H} }\right)\big| F\right] d\theta.
\label{99}
\end{equation}

\begin{proposition}\label{g}
Under the same hypothesis as in Theorem \ref{gaussian}, there exists $T_0>0$ such that, for all $t\in [0,T_0)$:
\[
C_1 \Phi(t)\leq g(F)\leq C_2 \Phi(t),\quad a.s.,
\]
where $C_1, C_2$ are positive constants depending on $b$  and $T_0$.
\end{proposition}

For the proof of this proposition, we need the following lemma. Recall that $\Phi(t):=\Phi(t,t,0)=E(X(t,x)^2)$.

\begin{lemma}\label{lemma}
There is a positive constant $C$ such that, for all $t>0$:
\begin{equation}   \label{6}
\sup_{\substack{ 0\leq r\leq t \\ y\in \mathbb{R}^{d}}}
E \left[ \left\| Du(r,y)\right\|  _{\mathcal{H}}| F \right]   \le C\sqrt{ \Phi(t)},
\end{equation}
and
\begin{equation} \sup_{\theta \geq 1}
\sup_{\stackrel{0\leq r\leq t}{y\in \mathbb{R}^d}} E\left[ E'\left(  \|\widetilde{D u(r,y)}\|_{\mathcal{H}} \right) \Big| F
\right] \leq C  \sqrt{\Phi(t)}, \qquad a.s.
\label{7}
 \end{equation}
 \end{lemma}

\begin{proof}
From (\ref{3}) and applying Minkowski inequality, we get:
\[
\left\| Du(t,y)\right\|  _{H}\leq   \sqrt{ \Phi (t)}+    \left\| b^{\prime }\right\|
_{\infty }\int_{0}^{t}\int_{\mathbb{R}^{d}}\ \left\| Du(s,y-z)\right\|
_{H}\Gamma (t-s,dz)ds.
\]%
As a consequence, we have the following estimate:
\begin{eqnarray*}
E\left[ \left\| Du(t,y)\right\|  _{H}|F\right] & \leq &   \sqrt{\Phi
(t)  }  +   \left\| b^{\prime }\right\|_{\infty }  \\
&&\times \int_{0}^{t}\int_{\mathbb{R}^{d}}\
E\left[ \left\| Du(s,y-z)\right\|  _{H}|F\right] \Gamma
(t-s,dz)ds.
\end{eqnarray*}
Set%
\[
Y_{t}:=\sup_{\substack{ 0\leq r\leq t \\ y\in \mathbb{R}^{d}}}E%
\left[ \left\| Du(r,y)\right\|  _{H}|F\right] .
\]%
Then, we have proved that
\[
Y_{t}\leq \sqrt{  \Phi (t)}+  \left\| b^{\prime }\right\|  _{\infty }\int_{0}^{t}\
Y_{s} \, \Gamma (t-s,\mathbb{R}^{d}) \, ds,
\]%
and a suitable generalisation of Gronwall lemma (e.g. \cite[Lemma 15]{dalang}) allows us to conclude the proof.
\end{proof}

\medskip

 \begin{proof}[Proof of Proposition  \ref{g}]
It follows the same lines as in the proofs of Propositions 4.5 and 5.5 in \cite{nq2}.
Indeed, one just needs to be careful and try to keep the sharpest bounds which appear throughout the proof.
Let us be more precise about this method, namely, by (\ref{3}) and (\ref{99}), we have that
$$g_F(F)= \Phi(t)+ \sum_{i=1}^3 A_i(t,x),$$
where
\begin{eqnarray*}
  A_1(t,x) &  =&   E\left[
  \int_0^t \int_{\mathbb{R}^d}
 b'(u(s,x-y)) \langle X(t,x), D u(s,x-y) \rangle_\mathcal{H}
 \Gamma(t-s,dy) ds    \Big| F\right] \\
A_2(t,x)& = &  \int_0^\infty e^{-\theta} E\left[ E'\left(
   \int_0^t \int_{\mathbb{R}^d} b'(\widetilde{u(s,x-y)}) \right. \right.  \\
& & \quad \times \left. \left. \langle  X(t,x)  , \widetilde{D u(s,x-y)} \rangle_\mathcal{H}   \Gamma(t-s,dy)
ds    \right) \Big| F\right] d\theta,  \\
A_3(t,x) &= &  \int_0^\infty e^{-\theta} E\Bigg[ E'\Bigg(    \int_{[0,t]^2}
 \int_{\mathbb{R}^{2d}}  b'(u(s,x-y)) b'(\widetilde{u(r,x-z)}) \\
&& \times    \langle Du(s,x-y) ,   \widetilde{D u(r,x-z)} \rangle_\mathcal{H}
 \Gamma(t-s,dy) \Gamma(t-r,dz)  dr ds    \Bigg) \Bigg| F\Bigg] d\theta.
\end{eqnarray*}
Let us first prove the lower bound. For this, observe that we have
$$g_F(F)\geq   \Phi(t)-|A_1(t,x)+A_2(t,x)+A_3(t,x)|.$$
Let us bound $|A_1(t,x)|$. We will use the following notation:
\begin{equation}
\Psi(t):= \int_0^t \Gamma(s,\mathbb{R}^d) ds.
\label{8}
\end{equation}
Notice that, by hypothesis, $\sup_{t\in [0,T]} \Psi(t)<+\infty$ and in fact $\Psi(t)$ converges
to $0$ as $t$ tends to $0$. We can write, applying (\ref{6}) in Lemma \ref{lemma} and using the notation (\ref{8}):
 \begin{eqnarray*}
\left| A_{1}(t,x)\right|  &\leq &\sqrt{\Phi (t)}\ \left\| b^{\prime }\right\|
_{\infty }\mathbb{E}\ \left[ \int_{0}^{t}\int_{\mathbb{R}^{d}}\left\|
Du(s,x-y)\right\|_{\mathcal{H}}\Gamma (t-s,dy)ds \Big| F\right]  \\
&\leq &\sqrt{\Phi (t)}\ \left\| b^{\prime }\right\| _{\infty }\left(
\int_{0}^{t}\Gamma (s,\mathbb{R}^{d})ds\right) \sup_{\substack{ 0\leq s\leq t
\\ z\in \mathbb{R}^{d}}}\mathbb{E}\left[ \left\| Du(s,z)\right\|
_{\mathcal{H}}|F\right]  \\
&\leq &C\ \Phi (t)\ \left\| b^{\prime }\right\| _{\infty }\left(
\int_{0}^{t}\Gamma (s,\mathbb{R}^{d})ds\right) \\
&=&  C\left\| b^{\prime }\right\| _{\infty }  \Phi(t) \Psi(t).
\end{eqnarray*}
For the analysis of the term $|A_2(t,x)|$, one can proceed using exactly the same arguments (the expectation $E'$ and
the integral with respect to $\theta$ do not affect the final result), but here we will need to apply (\ref{7}) in
Lemma \ref{lemma}. Indeed, one also obtains that
$$|A_2(t,x)| \leq  C \|b'\|_\infty \Phi(t) \Psi(t).$$
Eventually, in order to bound $|A_3(t,x)|$,  we observe that
\begin{eqnarray*}
\left| A_{3}(t,x)\right|  &\leq & \|b'\|^2_\infty  \int_0^\infty e^{-\theta} \int_{[0,t]^2} \int_{\mathbb{R}^{2d}} \\
& &  \quad \times E\left[ \|D u(s,x-y)\|_{\mathcal{H}}  E'\left( \| \widetilde{D u(r,x-z)} \|_{\mathcal{H}} \right) \Big| F\right] \\
& & \quad \times \Gamma(t-s,dy) \Gamma(t-r,dz) dr ds d\theta.
\end{eqnarray*}
for which we apply Cauchy-Schwartz inequality  with respect to  the conditional expectation with respect to $F$ and we
use the estimate
\[
\sup_{\theta \geq 1}
\sup_{\stackrel{0\leq r\leq t}{y\in \mathbb{R}^d}} E\left[ E'\left(  \|\widetilde{D u(r,y)}\|^2_{\mathcal{H}}  \; dr \right) \Big| F
\right] \leq C  \Phi(t) , \qquad a.s.,
 \]
whose proof is similar to that of (\ref{7}). In this way, we obtain that
 \[
  |A_3(t,x)| \leq   C  \|b'\|^2_\infty  \Phi(t)\Psi(t)^2 .
\]
Putting together the bounds for $|A_i(t,x)|$, $i=1,2,3$, we have that there are positive  constants $k_1$ and $k_2$
(only depending on   $b$) such that, for all $t\in (0,T]$ and $x\in \mathbb{R}^d$ (recall that $F=u(t,x)-E(u(t,x))$),
\begin{align*}
g_F(F) & \geq  k_1 \Phi(t) - k_2\left( \Phi(t)\Psi(t)+\Phi(t) \Psi(t)^2\right)\\
& =  \Phi(t) \left( k_1 - k_2\left( \Psi(t)+\Psi(t)^2\right)\right),
\end{align*}
where we remind that $\Psi(t)=\int_0^t \Gamma(s,\mathbb{R}^d) ds$. Hence, if $t$ is sufficiently small, say $t<T_0$ with $T_0>0$ satisfying
$k_1 - k_2\left( \Psi(T_0)+\Psi(T_0)^2\right)>0$, then
$$g(F)\geq C_1 \Phi(t),$$
where $C_1$ is a positive constant depending on $b$  and $T_0$. This proves the lower bound in the statement.

Concerning the upper bound, observe that we have, for all $t\in (0,T_0)$:
\begin{align*}
g_F(F) & \leq  \Phi(t) +\sum_{i=1}^3 |A_i(t,x)| \\
& \leq C \Phi(t) \left( 1 + \Psi(T_0)+\Psi(T_0)^2\right) \\
& \leq C_2 \Phi(t),
\end{align*}
where $C_2$ is a positive constant depending on $b$ and $T_0$. This concludes the proof.
\end{proof}


\section{Application to spatially homogeneous SPDEs}
\label{appl-spde}

As we commented in the Introduction, the main examples of equations of the form (\ref{2}) which Theorem \ref{gaussian} can be applied to correspond to
the mild formulation of the following SPDEs:
\begin{equation} \label{spde}
L u(t,x) =  b(u(t,x))+ \sigma \dot{W}(t,x), \qquad (t,x)\in [0,T]\times \mathbb{R}^d,
\end{equation}
where $L$ is a second order differential operator with constant coefficients and $\dot W$ is a Gaussian noise which is white in time and has some
spatially homogeneous correlation. More precisely, on some probability space
$(\Omega, \mathcal{F},P)$, we consider a family of centered Gaussian random variables $W = \{W(\varphi),\, \varphi \in
\mathcal{C}_0^\infty (\mathbb{R}^{d+1})\}$, where $\mathcal{C}_0^\infty (\mathbb{R}^{d+1})$ denotes the space of
infinitely differentiable functions with compact support, with the following covariance functional:
\begin{equation}
E(W(\varphi) W(\psi)) = \int_0^\infty \int_{\mathbb{R}^d} \left( \varphi(t)*{\psi_{(s)}}(t)\right) (x) \, \Lambda(dx)
dt, \quad \varphi, \psi\in \mathcal{C}^\infty_0(\mathbb{R}^{d+1}), \label{25}
\end{equation}
where $\psi_{(s)}(t,x):=\psi(t,-x)$ and $\Lambda$ is a non-negative and non-negative definite tempered measure.
For the right-hand side of (\ref{25}) to define a covariance functional, it turns out that
$\Lambda$ has to be symmetric and the Fourier transform of a non-negative tempered measure $\mu$ (see
\cite[Chapter VII, Th\'eor\`eme XVII]{sch}). Usually, $\mu$ is called the spectral measure of the noise $W$.
Let us denote by $(\mathcal{F}_t)_t$ the filtration generated by $W$, conveniently completed.

Usual examples of spatial correlations are given by measures of the form
$\Lambda(dx)=f(x)dx$, where $f$ is a non-negative and continuous
function on $\mathbb{R}^{d}\setminus \{0\}$, which is integrable in a
neighbourhood of $0$. For instance, one can consider a Riesz
kernel $f(x)=|x|^{-\epsilon}$, for $0<\epsilon<d$, while the
space-time white noise corresponds to consider $f=\delta_0$; in this latter case, the spectral measure
is the Lebesgue measure on $\mathbb{R}^d$.

By definition, a mild solution of (\ref{spde}) is a $\mathcal{F}_t$-adapted stochastic process
$\{u(t,x), \; (t,x)\in [0,T]\times \mathbb{R}^d\}$ satisfying
\begin{eqnarray}\label{eq:mild}
u(t,x) & =& \sigma \int_0^t \int_{\mathbb{R}^{d}} \Gamma(t-s,x-y) W(ds,dy) \nonumber\\
& & \qquad + \int_0^t \int_{\mathbb{R}^{d}} b(u(s,x-y)) \Gamma(t-s,dy) ds,
\end{eqnarray}
a.s. for all $(t,x)\in [0,T]\times \mathbb{R}^{d}$, where $\Gamma$ denotes the fundamental solution associated to $L$. 
Without any loose of generality, we can assume that $\sigma=1$. Let
us point out that the stochastic integral in the right-hand side of equation (\ref{eq:mild}) takes values in
$\mathbb{R}$ and is considered as an integral with respect to the cylindrical Wiener process associated to $W$, as
described in \cite[Section 3]{nq}. Note, however, that this stochastic integral can be also defined in the sense of
Dalang \cite{dalang} (see also \cite{walsh}) or even in the more abstract framework of Da Prato and Zabczyk \cite{dz}.

In this setting, the process $X$ is given by
\[
 X(t,x)= \int_0^t \int_{\mathbb{R}^{d}} \Gamma(t-s,x-y) W(ds,dy),
\]
for which one verifies that
\[
E(X(t,x)X(s,y))= \int_0^{s\land t} \int_{\mathbb{R}^d} e^{2\pi i (x-y)} \mathcal{F} \Gamma(t-r)(\xi)
\overline{\mathcal{F} \Gamma(s-r)(\xi)}\, \mu(d\xi) dr.
\]
Hence, in particular,
\[
 \Phi(t)= E(X(t,x)^2)= \int_0^t \int_{\mathbb{R}^d} | \mathcal{F} \Gamma(r)(\xi)|^2 \, \mu(d\xi) dr, \qquad t\in [0,T].
\]
At this point, we observe that $X(t,x)$ is a well-defined Gaussian random variable whenever the following condition is
satisfied (see e.g. \cite[Lemma 3.2]{nq}):
\begin{hypothesis}\label{h}
For all $t>0$, $\Gamma(t)$ defines a non-negative distribution with rapid decrease such that
\[
\int_0^T \int_{\mathbb{R}^{d}} |\mathcal{F} \Gamma(t)(\xi)|^2 \mu(d\xi) dt< \infty.
\]
Moreover, $\Gamma$ is a non-negative measure of the form $\Gamma(t, dx)dt$ such that, for all $T>0$,
$$\sup_{0\leq t\leq T} \Gamma(t,\mathbb{R}^{d})\leq C_T<\infty.$$
\end{hypothesis}
Thus, Theorem \ref{gaussian} applied to equation (\ref{eq:mild}) reads:
\begin{theorem}\label{gaussian-spde}
Assume that Hypothesis \ref{h} is satisfied and $b$ is of class $\mathcal{C}^1$ with a bounded derivative. Then, there
exists $T_0>0$ such that for all  $(t,x)\in (0,T_0)\times
\mathbb{R}^d$, the random variable $u(t,x)$ solution to (\ref{eq:mild}) has a density $p$ satisfying, for almost every
$z\in \mathbb{R}$:
\[
\frac{E\left| u(t,x)-m\right|}{ C_2 \Phi(t)} \exp\left\{ -\frac{(z-m)^2}{C_1 \Phi(t) } \right\}\leq p(z) \leq
\frac{E\left| u(t,x)-m\right|}{ C_1 \Phi(t)} \exp\left\{ -\frac{(z-m)^2}{C_2 \Phi(t)}\right\},
\]
where $m=E(u(t,x))$ and the constants $C_1, C_2$ are positive and only depend on $b$  and $T_0$.
\end{theorem}

\begin{example}
Theorem \ref{gaussian-spde} applies to the case of the stochastic heat equation in any space dimension and the
stochastic wave equation in $\mathbb{R}^d$ with $d\leq 3$. Indeed, in both examples, Hypothesis \ref{h} is fulfilled if
and only if
\[
\int_{\mathbb{R}^d} \frac{1}{1+|\xi|^2} \, \mu(d\xi) <+\infty
\]
(see e.g. \cite[Section 3]{dalang}). As a consequence, in the case of the stochastic wave equation, Theorem \ref{gaussian-spde} exhibits an improvement of 
\cite[Theorem 5.3]{nq2}, where the lower and upper Gaussian bounds were not optimal and a slightly
stronger condition on $\mu$ was assumed.
\end{example}

\begin{remark}
Independently of the results in \cite{nv}, the study of the existence and smoothness of the density for SPDEs of the form (\ref{spde}) 
(in particular for stochastic heat and wave equations), has already been tackled by several authors. Namely, let us mention the works
\cite{cn,MS,mms,QS,QS2,nq}.

\end{remark}

\end{document}